\documentclass[12pt,twoside]{article}
\usepackage{amsmath,amsfonts,amssymb,amsthm,array,mathdots}
\usepackage{a4,enumitem,indentfirst}
\usepackage{etoolbox} 
\usepackage{url} 

\usepackage{stmaryrd} 

\usepackage{epsfig, color} 

\graphicspath{{figures/}}

\theoremstyle{plain}

\newtheorem{theo}{Theorem}
\newtheorem{prop}{Proposition}[section]

\newtheorem{lemma}[prop]{Lemma}

\theoremstyle{definition}
\newtheorem{example}[prop]{Example}
\AtEndEnvironment{example}{\null\hfill$\diamond$}%

\newtheorem{remark}[prop]{Remark}
\AtEndEnvironment{remark}{\null\hfill$\diamond$}%

\newcommand\AP{\operatorname{AP}}
\newcommand\interp{\operatorname{interp}}

\newcommand{\bt}{{\mathbf t}}

\newcommand{\interior}{\operatorname{int}}

\newcommand{\ZZ}{{\mathbb{Z}}}
\newcommand{\QQ}{{\mathbb{Q}}}
\newcommand{\RR}{{\mathbb{R}}}

\newcommand{\Ss}{{\mathbb{S}}}

\newcommand{\TT}{{\mathbb{T}}}

\newcommand{\cL}{{\cal L}}

\newcommand{\cF}{{\cal F}}

\newcommand{\cG}{{\cal G}}

\begin{document}

\title{Stability of compact actions \\
and a result on divided differences}
\author{Carlos Gustavo Moreira
\and Nicolau C. Saldanha}
\date{}

\maketitle

\begin{abstract}
We study smooth locally free actions of $\RR^n$
on manifolds $M$ of dimension $n+1$.
We are interested in compact orbits and
in compact actions: actions with all orbits compact.
Given a compact orbit in a neighborhood of compact orbits,
we give necessary and sufficient conditions for the existence
of a $C^k$ perturbation with noncompact orbits in the given neighborhood.
We prove that if such a perturbation exists
it can be assumed to differ from the original action
only in a smaller neighborhood of the initial orbit.
As an application, for each $k$,
we give examples of compact actions
which admit $C^{k-1}$-perturbations with noncompact orbits
but such that
all $C^k$-perturbations are compact.
The main result generalizes for $k > 1$
a previous result for the case $C^1$.

A critical auxiliary result is
an estimate on divided differences.
\end{abstract}

\medskip 


\section{Introduction}
\label{section:intro}

For a positive integer $n$,
let $\TT^n = H/L$
where $H \approx \RR^n$ and
$L \subset H$ is a lattice, $L \approx \ZZ^n$.
Let $M = \TT^n \times (-\epsilon,\epsilon)$,
with coordinates $(y_1,\ldots,y_n,z)$.
The coordinates $y_i$ are called \textit{horizontal};
the coordinate $z$ is \textit{vertical}.
Let $Y_1, \ldots, Y_n, Z$ be the corresponding
tangent vector fields to $M$.
The tori with equation $z = z_0$ are horizontal;
the horizontal torus $z=0$ is the \textit{base} torus.

Let $D \approx \RR^n$ be a real vector space
with coordinates $x_1, \ldots, x_n$.
Let $a_{ij}: (-\epsilon,\epsilon) \to \RR$ 
($1 \le i, j \le n$) be smooth real functions;
let $A: (-\epsilon,\epsilon) \to \RR^{n\times n}$ be such that
$A(z)$ is a matrix with entries $a_{ij}(z)$.
We assume that the matrices $A(z)$ are invertible.
Thus, $A$ is a family (indexed by $z \in (-\epsilon,\epsilon)$)
of invertible linear transformations $A(z)$
from $D$ (with coordinates $x_j$)
to $H$ (with coordinates $y_i$).
The family of tangent vector fields
\begin{equation}
\label{eq:a}
X_j(y_1,\ldots,y_n,z) = \sum_{1 \le i \le n} a_{ij}(z) Y_i,
\qquad 1 \le j \le n 
\end{equation}
defines a smooth locally free action $\theta$ of $D$ on $M$.
Notice that the Lie bracket $[X_{j_0},X_{j_1}]$ is equal to zero
(for all $j_0$, $j_1$).
The vector space $D$ is the \textit{domain} of the action $\theta$.
We call such an action \textit{homogeneous horizontal}.
{Let $\cL(D;H) \approx \RR^{n\times n}$
be the vector space of linear transformations from $D$ to $H$,
so that we have $A: (-\epsilon,\epsilon) \to \cL(D,H)$.}
For a positive integer~$j$,
$A^{(j)}: (-\epsilon,\epsilon) \to \cL(D;H)$
is the $j$-th derivative of $A$
with respect to the variable $z \in (-\epsilon,\epsilon)$.

\begin{remark}
\label{remark:Prop10}
This may seem to be a very special situation,
but that is not the case.
Given a smooth locally free action of $\RR^n$ 
on an arbitrary smooth manifold $M$ of dimension $n+1$,
if an open neighborhood of a compact orbit
also consists of compact orbits
then Proposition~1.0 from \cite{stable}
essentially reduces the problem to the above situation.
\end{remark}

Given a homogeneous horizontal action $\theta$,
for a positive integer $k$ and $\delta > 0$,
a $\delta$-$C^k$-close family $(\tilde X_j)$
of commuting tangent vector fields satisfies:
$[\tilde X_{j_0},\tilde X_{j_1}] = 0$ (for all $j_0$, $j_1$)
and
$\|\tilde X_j - X_j\|_{C^k} < \delta$ (for all $j$).
We assume the vector fields $\tilde X_j$ to be smooth.
Actually, we only need them to be of class $C^{k+1}$;
on the other hand, we do not need uniform estimates on the $C^{k+1}$ norm
of the perturbed vector fields.
Such a family describes a local action of $D$ on $M$.
If the vector fields $\tilde X_j$ can be written in the form
\begin{equation}
\label{eq:tildea}
\tilde X_j(y_1,\ldots,y_n,z) =
\tilde a_{Z,j}(z) Z + \sum_{1 \le i \le n} \tilde a_{ij}(z) Y_i 
\end{equation}
where the coefficients $\tilde a_{Z,j}, \tilde a_{ij}$
are functions of $z$ only
(and not of $y_1, \ldots, y_n$)
then $\tilde\theta$ is called \textit{homogeneous}.

\bigbreak

We are ready to state the main result of the present paper.

\begin{theo}
\label{theo:ck}
Consider a positive integer $k$ and
a locally free homogeneous horizontal action $\theta$ of $D = \RR^n$ on
$M = \TT^n \times (-\epsilon,\epsilon)$.
Construct
{$A: (-\epsilon,\epsilon) \to \RR^{n\times n}$} as above.
The following conditions are equivalent:
\begin{enumerate}
\item{There exists a subspace $D' \subset D$ of codimension $1$
such that, for all $w \in D'$ and all $j$, $1 \le j \le k$, we have
$(A^{(j)}(0))w = 0$.}
\item{For all $\delta > 0$ there exists a $\delta$-$C^k$-close
family $\tilde\theta$ of commuting vector fields
and a point in the base torus which belongs to a noncompact orbit.}
\item{For all $\delta \in (0,\epsilon/2)$ there exists
a homogeneous $\delta$-$C^k$-close
family $\tilde\theta$ of commuting vector fields
coinciding with $\theta$ in
$M \smallsetminus (\TT^n \times (-\delta,\delta)) \subset M$
and for which every point in the base torus belongs to a noncompact orbit.}
\end{enumerate}
\end{theo}

{In the above statement, 
the perturbed action $\tilde\theta$ can be taken to be of class
either $C^{k+1}$ or $C^\infty$.}

Theorem~\ref{theo:ck}~generalizes the main result of \cite{stable},
which is essentially equivalent to the case $k=1$
of the statement above.
The base torus is a $C^k$-stable orbit if it satisfies
none of the conditions in Theorem~\ref{theo:ck}
(and $C^k$-unstable otherwise).
The concept extends to orbits of compact actions.
Notice that the base torus is always $C^0$-unstable
(Proposition~1.6 in \cite{stable}).

\begin{remark}
\label{remark:b}
We can rewrite Equation~\eqref{eq:a} as
\begin{equation}
\label{eq:b}
Y_i = \sum_{1 \le j \le n} b_{ji}(z) X_j(y_1, \ldots, y_n, z), \quad
B(z) = (b_{ji})_{1 \le i,j \le n} = (A(z))^{-1}.
\end{equation}
Here, $B: (-\epsilon,\epsilon) \to \RR^{n \times n} = \cL(H;D)$
is a smooth path and
$B(z) = (A(z))^{-1}$ is a linear transformation from $H$ to $D$.
Condition 1 in Theorem~\ref{theo:ck} can be equivalently stated
in terms of $B$ and $H$ as follows:

\textit{There exists a subspace $H' \subset H$ of codimension $1$
such that, for all $v \in H'$ and all $j$, $1 \le j \le k$, we have
$(B^{(j)}(0))v = 0$.}

For a perturbed action $\tilde\theta$ we may also write
\begin{equation}
\label{eq:bt}
Y_i(p) = c_i(p) Z + \tilde Y_i(p), \qquad
\tilde Y_i(p) = 
\sum_{1 \le j \le n} \tilde b_{ji}(p) \tilde X_j(p),
\end{equation}
defining
$\tilde B = (\tilde b_{ji})_{1 \le i,j \le n}:
\TT^n \times (-\epsilon,\epsilon) \to
\RR^{n \times n} = \cL(H;D)$.
\end{remark}

\begin{example}
\label{example:dance}
Let $\varphi: (-\epsilon,\epsilon) \to \RR$ be a smooth function.
Take $n=2$ and
\begin{equation}
\label{equation:dance}
a_{11}(z) = a_{22}(z) = \cos(\varphi(z)),
\qquad
a_{12}(z) = -a_{21}(z) = \sin(\varphi(z)).
\end{equation}
Define $X_1$ and $X_2$ as in Equation \eqref{eq:a};
see Figure~\ref{fig:dance}.
Let $k \ge 1$ be such that
$\varphi^{(k)}(0) \ne 0$ and
$\varphi^{(j)}(0) = 0$ for all $j$, $1 \le j < k$.
Then the base torus is both $C^k$-stable 
and $C^{k-1}$-unstable.
Example~1.4 in \cite{stable} considers the case $\varphi(z) = 2\pi z$
(for which the base torus is $C^1$-stable).
\end{example}

\begin{figure}[ht]
\centering
\def\svgwidth{\textwidth}
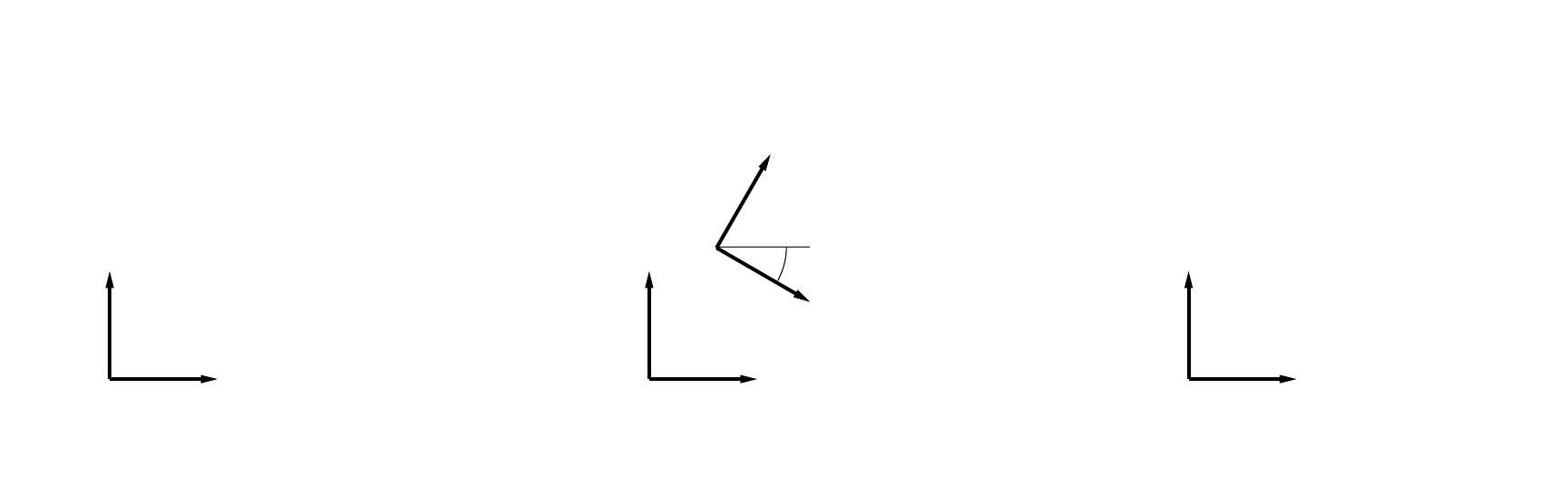
\caption{A compact action of $\RR^2$ on $\TT^2 \times (-\epsilon,\epsilon)$.}
\label{fig:dance}
\end{figure}

More generally, a locally free action of $\RR^n$
on $M^{n+1}$ is \textit{compact} if all its orbits are compact
(and therefore tori, and perhaps a few Klein bottles).
As we have seen in Remark~\ref{remark:Prop10},
up to minor adjustments,
appropriate coordinates in a neighborhood of an orbit
turn a compact action into a homogeneous horizontal action.
The concept of $C^k$ stability thus applies to orbits of a compact action.

\bigskip

\begin{example}
\label{example:dance2}
We consider a more global variant of Example~\ref{example:dance}.
For fixed $m \in \ZZ$, 
let $\varphi: \RR \to \RR$ be a smooth function
satisfying $\varphi(z+1) = \varphi(z)+2\pi m$ (for all $z \in \RR$).
Define a compact action $\theta$ of $\RR^2$ on $\TT^3 = \RR^3/\ZZ^3$
(with coordinates $y_1, y_2, z$)
by Equation~\eqref{equation:dance}.
It follows from \cite{stable} that
if $\varphi$ has no critical points then $\theta$ is $C^1$ stable
but not $C^0$ stable.
More precisely, there exists $\delta > 0$ such that:
if $\tilde\theta$ is $\delta$-$C^1$-close to $\theta$
then $\tilde\theta$ is a compact action.

More generally, assume that $\varphi$ has
at least one critical point of order $k$
and no critical point of higher order.
It follows from Theorem~\ref{theo:ck}
that
$\theta$ is $C^k$-stable but not $C^{k-1}$-stable.
For instance, if $\varphi(z) = \cos(2\pi z)$ then
$\theta$ is $C^2$-stable but not $C^1$-stable.
If $\varphi(z) = \cos^3(2\pi z)$ then
$\theta$ is $C^3$-stable but not $C^2$-stable.
\end{example}

\bigskip

\begin{remark}
\label{remark:generic}
{For $n > 2$ and $\theta$ a generic compact action,
all orbits are $C^1$-stable (Example~1.5 in \cite{stable}).
For $n = 2$, this is not the case.
Take for example $A: \RR/\ZZ \to \RR^{2\times 2}$,
\[ a_{11}(z) = 2+\sin(2\pi z), \quad a_{22}(z) = 2-\cos(2\pi z), \quad
a_{12}(z) = a_{21}(z) = 0 \]
so that $\det(A'(z)) = 2\pi^2\sin(4\pi z)$.
The base torus is $C^2$-stable but not $C^1$-stable.
Moreover, for any $C^2$ nearby horizontal homogeneous
$\tilde A: \RR/\ZZ \to \RR^{2\times 2}$,
there exists $\tilde z_0 \approx 0$
with $\det(\tilde A'(\tilde z_0)) = 0$:
the torus $z = \tilde z_0$ is therefore not $C^1$-stable.}

{On the other hand:
for $n = 2$ and a  generic compact action,
all orbits are $C^2$-stable.
Indeed, we may assume by genericity
that $\det(A'(z))$ has isolated zeroes.
Let $z_0$ be one such value of $z$.
Again by genericity, we may assume that $A'(z_0)$ has rank $1$
and that the intersection of the kernels of
$A'(z_0)$ and $A''(z_0)$ is $0$,
as in the previous example.
The torus $z = z_0$ is then $C^2$-stable (but not $C^1$-stable).}
\end{remark}

\bigbreak

As for the case $k=1$ (discussed in \cite{stable}),
the implications $1 \to 3$ and $3 \to 2$ are not hard.
Indeed, most of \cite{stable} is dedicated to proving $2 \to 1$
in the case $k=1$.
The proof of the general case follows a similar strategy.
The main difference is that,
where in \cite{stable} we use the mean value theorem,
in the case $k > 1$ we need a subtler result
concerning divided differences:

\bigbreak

\begin{prop}
\label{prop:fh}
Let $k \ge 2$, $\epsilon > 0$,
$0 < B_k \le B_{k+1}$ and $f_0(z) = z$.
Then there exists $C_k > 1$
(depending only on $k$, $\epsilon$ and $B_{k+1}$)
such that,
for all $h_0 \in C^{\infty}([-\epsilon,+\epsilon];\RR)$,
for all $f, h \in C^k([-\epsilon,+\epsilon];\RR)$ and for all $\eta > 0$,
if $B = h_0^{(k)}(0)$ and
\[
\begin{gathered}
\|h_0\|_{C^{k+1}} \le B_{k+1}, \qquad
|B| \ge B_k, \\
0 < \eta < \frac{\min\{1,B_k\}}{2C_k}, \quad
\|h-h_0\|_{C^k} < \eta, \quad
\|f-f_0\|_{C^k} < \eta, \quad f(0) > 0
\end{gathered}
\]
then the following property holds.

Set $z_j = f^j(0)$;
for $\tilde z_0 \in (z_{-2k},z_{2k})$,
set $\tilde z_j = f^j(\tilde z_0)$.
We then have
\[ \frac{B-C_k{\eta}}{k!}  <
[z_0,z_1,\ldots,z_k;h(\tilde z_0),h(\tilde z_1),\ldots,h(\tilde z_k)]
< \frac{B+C_k{\eta}}{k!}. \]
In particular, the divided difference
$[z_0,z_1,\ldots,z_k;h(\tilde z_0),h(\tilde z_1),\ldots,h(\tilde z_k)]$
has the same sign as $B$ and is therefore nonzero.
\end{prop}

{
Since $C_k$ is chosen after $k$ and $\epsilon$,
we can force $\eta$ to be much smaller than $\epsilon/k$.
Together with $\|f - f_0\|_{C^k} < \eta$,
this implies that both $z_j = f^j(0)$ and $\tilde z_j = f^j(\tilde z_0)$
are well defined and belong to the interval $(-\epsilon,\epsilon)$
(for all $j$ with $|j| \le 2k$).}

\bigbreak

In Section~\ref{section:dd} we revise divided differences
and prove Proposition~\ref{prop:fh}.
Section~\ref{section:proofck} begins with
the proofs of the easy implications
in Theorem~\ref{theo:ck} (i.e., $1 \to 3$ and $3 \to 2$).
After reviewing several results and constructions from \cite{stable}
we complete the proof 
of Theorem~\ref{theo:ck} by proving the remaing implication $2 \to 1$.

This paper can be considered a sequel of \cite{stable}.
Other works related to that paper are
\cite{BegazoSaldanha2004, BegazoSaldanha2005, Ccoyllo}.

\bigskip

Support from CNPq, CAPES and Faperj (Brazil) 
is gratefully acknowledged.
The authors thank Dimitar Dimitrov 
for helpful conversations about divided differences
and the referee for several careful remarks.

\bigbreak


\section{Divided differences}
\label{section:dd}

Consider real numbers $x_0 < x_1 < \cdots < x_k$ and 
$y_0, y_1, \ldots, y_k$.
There exists a unique real polynomial
\begin{equation}
\label{equation:interp}
\begin{aligned}
P(X) &= c_k X^k + \cdots + c_j X^j + \cdots + c_1 X + c_0 \\
&= \interp([x_0,\ldots,x_k],[y_0,\ldots,y_k],X) 
\end{aligned}
\end{equation}
of degree at most $k$
with the property that $P(x_i) = y_i$ (for all $i$, $0 \le i \le k$).
The coefficient of degree $k$ is written as
\[ c_k = [x_0,x_1,\ldots,x_k;y_0,y_1,\ldots,y_k]. \]
The above expression is a \textit{divided difference}
and there is a substantial literature about it
(see \cite{deBoor} for an introduction and 
an extensive list of references
and \cite{LMS} for similar estimates with a different application).
Sometimes the variables $x_i$ are left implicit:
\( c_k = [y_0,y_1,\ldots,y_k] \).
In this notation we have
\begin{equation}
\begin{aligned}
P(X) &= [y_0] + [y_0,y_1] (X-x_0) + [y_0,y_1,y_2](X-x_0)(X-x_1) + \cdots \\
& \qquad \cdots + [y_0,y_1,\ldots,y_k] (X-x_0)(X-x_1)\cdots(X-x_{k-1}).
\end{aligned}
\end{equation}

If $f: [x_0,x_k] \to \RR$ is a function of class $C^k$
satisfying $f(x_i) = y_i$ (for all $i$, $0 \le i \le k$)
then there exists $\zeta \in [x_0,x_k]$ with
\begin{equation}
\label{equation:zeta}
\frac{1}{k!} f^{(k)}(\zeta) = [y_0,y_1,\ldots,y_k].
\end{equation}
Another formula is
\( [y_0,y_1,\ldots,y_k] = {\det(W)}/{\det(V)} \)
where $V$ and $W$ are the Vandermonde and almost-Vandermonde matrices
\begin{equation}
\label{equation:VW}
V = \begin{pmatrix}
1 & 1 & \cdots & 1 \\
x_0 & x_1 & \cdots & x_k \\
\vdots & \vdots & & \vdots \\
x_0^{k-1} & x_1^{k-1} & \cdots  & x_k^{k-1} \\
x_0^{k} & x_1^{k} & \cdots & x_k^{k} 
\end{pmatrix}, \quad
W = \begin{pmatrix}
1 & 1 & \cdots  & 1 \\
x_0 & x_1 & \cdots  & x_k \\
\vdots & \vdots &  & \vdots \\
x_0^{k-1} & x_1^{k-1} & \cdots  & x_k^{k-1} \\
y_0 & y_1 & \cdots &  y_k 
\end{pmatrix}.
\end{equation}
If we expand $\det(W)$ along the $k$-th row
and use the formula for the determinant of a Vandermonde matrix
we obtain
\begin{equation}
\label{equation:expand}
[y_0,y_1,\ldots,y_k] = \sum_{0 \le j \le k} u_j y_j, \qquad
u_j = \prod_{i \ne j} \frac{1}{x_i - x_j}.
\end{equation}

We are going to study cases when $x_0 < x_1 < \cdots < x_k$
is almost an AP (arithmetic progression).
More precisely, for a given $C_{\AP} \in (0,1/4)$,
$(x_j)$ is an increasing \textit{$C_{\AP}$-quasi-AP}
if $x_1 > x_0$ and, for all $j$, $0 < j \le k$,
\[ (1-C_{\AP}) (x_1 - x_0) < x_j - x_{j-1} < (1+C_{\AP}) (x_1 - x_0).\]

\bigbreak

\begin{remark}
\label{remark:1+a}
We have $(1+a)^m<(1-a)^{-m}<1+(m+1)a$ if $0<a<1/m(m+1)$.
Indeed, $(1-a)^m>1-ma$, and thus, in this case,
$(1-a)^m(1+(m+1)a)>(1-ma)(1+(m+1)a)=1+a-m(m+1)a^2>1$.
\end{remark}

\begin{lemma}
\label{lemma:quasiAP}
Let $k, r$ be positive integers and $f_0(z) = z$.
Assume that $\epsilon > 0$ and that
$f:[-\epsilon,+\epsilon]\to\RR$ satisfies
$\|f-f_0\|_{C^1} < \eta$.
Let $c\in [-\epsilon,0]$ such that $f(c)>c$. 
Then, if $\eta\le\min\left\{ \frac{\epsilon}{kr}, \frac1{(kr)^2} \right\}$,
the sequence $(x_j)_{1\le j \le k}$, $x_j=f^{rj}(c)$,
is a $kr\eta$-quasi-AP.
\end{lemma}

\begin{proof}
For all $x \in [-\epsilon,\epsilon]$ we have
$|f(x) - x| < \eta \le \epsilon/(kr)$.
We thus have $f^{rj}(c) < f^{r(j+1)}(c) < f^{rj}(c) + \epsilon/k$
and therefore (by induction on $j$)
$x_j \in [-\epsilon,\epsilon]$ for all $j$, $0 \le j \le k$.

We have $0 < 1-\eta<f'(x)<1+\eta$ for every $x\in [-\epsilon,+\epsilon]$,
so $(1-\eta)^s<(f^s)'(x)<(1+\eta)^s$ for $0<s\le rk$, and thus 
{(by the mean value theorem)}
\[ \frac{x_j-x_{j-1}}{x_1-x_0}=
\frac{f^{r(j-1)}(f^r(c))-f^{r(j-1)}(c)}{f^r(c)-c}
\in ((1-\eta)^{r(j-1)},(1+\eta)^{r(j-1)}). \]
Remark~\ref{remark:1+a} implies 
\( ((1-\eta)^{r(j-1)},(1+\eta)^{r(j-1)}) \subset (1-kr\eta,1+kr\eta)\),
completing the proof.
\end{proof} 

\begin{lemma}
\label{lemma:interp}
Given $k > 0$, $\epsilon\in (0,1)$, for $\tilde C_k=(\frac{2(1+\epsilon)}{1-\epsilon})^k$, we have the following property.

Let $k\ge 2$, $x_0 < \cdots < x_k$ be an increasing $\epsilon$-quasi-AP;
let $P$ be a real polynomial of degree at most $k$.
If $|P(x_j)| \le C$ for all $j$ ($0 \le j \le k$) then
$|P(x)| \le C\tilde C_k$ for all $x \in [x_0,x_k]$.
\end{lemma}

\begin{proof}
Let $d=x_1-x_0>0$. By Lagrange interpolation, we have 
$$P(x)=\sum_{j=0}^k P(x_j) \frac{(x-x_0)\cdots \widehat{(x-x_j)}\cdots (x-x_k)}{(x_j-x_0)(x_j-x_1)\cdots (x_j-x_k)}.$$ 
In each term we have $|P(x_j)|\le C$, the absolute value of the numerator is at most $(1+\epsilon)^kd^kk!$ and the absolute value of the denominator of the $j$-th term is at least $(1-\epsilon)^kd^kj!(k-j)!$, so the absolute value of $P(x)$ is at most 
$$\left(\frac{1+\epsilon}{1-\epsilon}\right)^k\sum_{j=0}^k{k\choose j}=\left(\frac{2(1+\epsilon)}{1-\epsilon}\right)^k,$$
{as desired.}
\end{proof}

\bigbreak

\begin{lemma}
\label{lemma:tildehj}
Let $k \ge 2$, $f_0(z) = z$.
Then,
for all $h_0 \in C^{k+1}([-\epsilon,+\epsilon];\RR)$,
for all $f, h \in C^k([-\epsilon,+\epsilon];\RR)$ and for all $\eta > 0$,
if 
\[
\begin{gathered}
\|h_0\|_{C^{k+1}} \le B_{k+1}, \\
0 < \eta < \min\left\{ \frac{\epsilon}{4k}, \frac1{(k+1)2^{k+3}} \right\},
\quad
\|h-h_0\|_{C^k} < \eta, \quad
\|f-f_0\|_{C^k} < \eta
\end{gathered}
\]
then the following property holds:
for every integer $j$ with $-2k\le j\le 2k$,
if $\tilde h_{j}:=h\circ f^j$, then 
the natural domain of $\tilde h_{j}$ contains
$[-\epsilon/2,\epsilon/2]$ and we have

$$\|\tilde h_{j}-h_0\|_{C^k([-\epsilon/2,\epsilon/2];\RR)}
<((2^{k+3}(k+3)-2k)B_{k+1}+4)\eta.$$
\end{lemma}

Before we go into the proof,
we present some of the main tools.

\begin{remark}
\label{remark:faa}
In the proof of Lemma~\ref{lemma:tildehj},
we will use several times the combinatorial version 
of the Fa\`a di Bruno's formula:
\begin{equation}
\label{eq:Faa}
(F\circ G)^{(r)}(x)=
\sum_{\pi\in \Pi}F^{(|\pi|)}(G(x))\prod_{B\in \pi}G^{(|B|)}(x),
\end{equation}
where $\pi$ runs through the set $\Pi$ of all partitions
of the set $\{1, 2,\dots, r\}$
{($|\pi|$ is the number of sets in the partition $\pi \in \Pi$
and $|B|$ is the cardinality of $B \in \pi$)}.
\end{remark}

\begin{remark}
\label{remark:stirling2k}
Let  ${n\brace k}$ be the \textit{Stirling number of the second kind}:
the number of partitions
of a set with $n$ elements in $k$ nonempty subsets.
Recall that ${n\brace k} \le  {n\choose k}k^{n-k}$:
we may choose the smallest elements of each of the $k$ subsets,
and the other elements are distributed by some function
sending the remaining $n-k$ elements to the $k$ subsets. 
\end{remark}

\begin{proof}[Proof~of~Lemma~\ref{lemma:tildehj}]
The estimate $\|f-f_0\|_{C^k} < \epsilon/(4k)$
implies $x -|j| \epsilon/(4k) < f^{j}(x) < x+|j|\epsilon/(4k)$
for all integer $j$, $0 < |j| \le 2k$.
Thus, the functions $f^j$ and $\tilde h_j$
are well defined in $[-\epsilon/2,\epsilon/2]$.

In the cases below where we will apply Remark~\ref{remark:faa},
the map $G$ will always be $C^r$-close to the identity, i.e.,
it satisfies $\|G-f_0\|_{C^r} < \tilde\eta$,
for some small $\tilde\eta$.
For a given partition $\pi$,
let $s = |\pi|$ be the number of nonempty sets in the partition
and $m \le s$ be the number of singletons in $\pi$.
In order to estimate the sum in Equation~\eqref{eq:Faa},
we isolate the terms
$F^{(r)}(G(x))(G'(x))^r$
(corresponding to $s=r$, the partition into $r$ singletons)
and $F'(G(x))G^{(r)}(x)$
(corresponding to $s=1$, the partition into a single set).
We proceed to obtain estimates for the \textit{intermediate terms},
for which $1 < s < r$ and $t = s-m > 0$ subsets
in the partition $\pi$ which are not singletons.

Let $M=\max_{2\le \ell<r}\|F^{(\ell)}\|_{C^0}\le \|F\|_{C^r}$.
We consider separately the cases $t = 1$ and $t \ge 2$.
In the case $t = s-m=1$,
we have ${r\choose m}$ ways of choosing the singletons,
and the remaining $r-m$ elements form the remaining subset
(we should have $m<r-1$).
The corresponding sum can be bounded in absolute value by
\[ \sum_{m<r-1}{r\choose m}M(1+\tilde\eta)^m\tilde\eta
< M(2+\tilde\eta)^r\tilde\eta. \]
In the cases with $t = s-m\ge 2$,
the sum can be estimated using the inequality
${n\brace k} \le  {n\choose k}k^{n-k}$.
The number of these terms for fixed values of $m$ and $s$
is less than or equal to
$${r\choose m}{r-m \brace t}\le
{r\choose m}{r-m\choose t}t^{r-m-t}={r\choose t}{r-t\choose m}t^{r-t-m}.$$
The sum of the corresponding terms for a fixed value of
$t \ge 2$ can be bounded by 
\begin{align*}
\sum_{m<r-t}{r\choose t}{r-t\choose m} t^{r-t-m}
M (1+\tilde\eta)^m {\tilde\eta}^t
&=
M {r\choose t} t^{r-t} {\tilde\eta}^t
\sum_{m<r-t}
{r-t\choose m}
\left(\frac{1+\tilde\eta}{t}\right)^m \\
&<
M{r\choose t}\left(1+\frac{1+\tilde\eta}t\right)^{r-t}t^{r-t}{\tilde\eta}^t.
\end{align*}

Let $w(t)=t^r{\tilde\eta}^t$.
Notice that, provided $\tilde\eta\le e^{-r/2}$,
the map $w$ is decreasing in $t$ for $2\le t\le r$.
Indeed, the derivative
of $\log w(t)$ is $r/t+\log(\tilde\eta)\le r/2-r/2=0$.
In particular, $t^r{\tilde\eta}^t\le 2^r{\tilde\eta}^2$ for $t\in [2,r]$.
Since $t^{r-t}\le 2^{-t} t^r$ and
$1+\frac{1+\tilde\eta}t\le 1+\frac{1+\tilde\eta}2$ for $t\ge 2$,
the above sum can be bounded in absolute value by 
\begin{align*}
M\sum_{t\le r}{r\choose t}
\left(1+\frac{1+\tilde\eta}2\right)^{r-t}
\left(\frac12\right)^{t} 2^{r}{\tilde\eta}^2
&=M 2^{r}{\tilde\eta}^2\left(1+\frac{1+\tilde\eta}2+\frac12\right)^r \\
&=M(4+\tilde\eta)^r\tilde\eta^2.
\end{align*}

Therefore, assuming $\tilde\eta\le 2^{-r}<e^{-r/2}$,
the sum of the absolute values of the intermediate terms
is at most
\begin{equation}
\label{eq:intermediate}
M(2+\tilde\eta)^r\tilde\eta+M(4+\tilde\eta)^r\tilde\eta^2
<2M(2+\tilde\eta)^r\tilde\eta.
\end{equation}

Let us prove now by induction on $j$ that 
\begin{equation}
\label{eq:inductionj}
\|f-f_0\|_{C^k} < \eta 
\implies \|f^j-f_0\|_{C^k} < 2j\eta, \qquad
0 < j\le 2k;
\end{equation}
the case $j=1$ is trivial.
We first consider the $C^0$ and $C^1$ estimates.
We have
\[ |f^m(x)-x|=|\sum_{j=0}^{m-1}(f(f^j(x))-f^j(x))|\le
\sum_{j=0}^{m-1}|f(f^j(x))-f^j(x)|<m\eta\]
and, assuming $0<\eta<\frac1{2k(2k+1)}$
(and using Remark~\ref{remark:1+a})
we always have,
for $0\le m\le 2k$,
\[ (f^m)'(x)\in [(1-\eta)^m,(1+\eta)^m]\subset
(1-(m+1)\eta,1+(m+1)\eta). \]
In order to make the recursive estimate for $\|f^j-f_0\|_{C^k}$, we may use the previous estimate with $1<r\le k$, $F=f^j$, for some $0\le j<2k$ and $G=f$, and the fact that $M\le 2j\eta$ in this case.
Provided $\eta\le 2^{-k}$, we have
$$|(f^{j+1})^{(r)}(x)|=|(F\circ G)^{(r)}(x)|\le 2j\eta(1+\eta)^r+\eta(1+\eta)^j+4j(2+\eta)^r\eta^2.$$

So, assuming $\eta\le\frac1{(k+1)2^{k+3}}<\frac1{2k(2k-1)}$, since we have $(1+\eta)^r\le (1+\eta)^k<1+(k+1)\eta$, $(2+\eta)^r\le (2+\eta)^k=2^k(1+\eta/2)^k<2^k(1+(k+1)\eta/2)$ and $(1+\eta)^j\le (1+\eta)^{2k-1}<1+2k\eta$, we get
\[ \begin{gathered}
2j\eta(1+\eta)^r+\eta(1+\eta)^j+4j(2+\eta)^r\eta^2< \\
< 2j\eta+2j(k+1)\eta^2+\eta+2k\eta^2+4(2k-1)2^k(1+(k+1)\eta/2)\eta^2
\end{gathered} \]
and, since $(2k-1)(1+(k+1)\eta/2)<(2k-1)(1+\frac{k+1}{2^{k+3}})<2k-1+\frac{k(2k+1)}{2^{k+3}}<2k-1/2$,
this is at most
\[ \begin{gathered}
(2j+1)\eta+(2(2k-1)(k+1)+2k+2^{k+1}(4k-1))\eta^2<\\
< (2j+1)\eta+2^{k+3}(k+1)\eta^2\le 2(j+1)\eta, 
\end{gathered} \]
concluding the proof by induction of the claim in
Equation~\eqref{eq:inductionj}.

Now we will prove that if $0<\eta<\frac1{(k+1)2^{k+3}}$
and $\|f-f_0\|_{C^k} < \eta$ then
\[ \|f^{-1}-f_0\|_{C^k} <
\eta+\frac{2^{k+2}(k+1)}{k}\eta^2\le \left(1+\frac1{2k}\right)\eta<2\eta. \]
Indeed, taking $F=f^{-1}$ and $G=f$, we have $F(G(x))=x$,
so $|F(y)-y|=|G(F(y))-F(y)|<\eta$ and
$F'(G(x))=1/G'(x)\in (\frac1{1+\eta},\frac1{1-\eta})$,
so $|F'(G(x))-1|<\frac{\eta}{1-\eta}=\eta+\frac{\eta^2}{1-\eta}<
(1+\frac1{2k})\eta$ (thus taking care of the $C^0$ and $C^1$ estimates).

The $C^r$ estimates for $r \ge 2$ are proved by induction on $r$.
We have
$0=(F\circ G)^{(r)}(x)=F^{(r)}(G(x))(G'(x))^r+F'(G(x))G^{(r)}(x)+S$,
where $S$ is the sum of the intermediate terms in Fa\`a di Bruno's formula,
so \[ F^{(r)}(G(x))=-(G'(x))^{-r}(F'(G(x))G^{(r)}(x)+S). \]
We have $|G^{(r)}(x)|<\eta$ and
$|(G'(x))^{-r}F'(G(x))|<(1-\eta)^{-(r+1)}<1+(r+2)\eta$.
Also, by induction hypothesis,
$|F^{(s)}(y)|<2\eta$ for $2\le s<r$.
We then have
(by Equation~\eqref{eq:intermediate},
using the estimates $M < 2\eta$ and $\tilde\eta = \eta$)
\[ |S|<4\eta(2+\eta)^r\eta=4(2+\eta)^r\eta^2=2^{r+2}(1+\eta/2)^r\eta^2. \]
We therefore have
{%
\[ |(G'(x))^{-r}S|
<2^{r+2}(1+\eta/2)^r(1-\eta)^{-r}\eta^2<2^{r+2}(1+(2r+2)\eta)\eta^2 \]}
and
thus 
\[ |F^{(r)}(G(x))|\le (1+(r+2)\eta)\eta+2^{r+2}(1+(2r+2)\eta)\eta^2<\eta+\frac{2^{k+2}(k+1)}{k}\eta^2, \] which implies our claim.

It follows that, for $-2k\le j\le 2k$, $|f^j(x)-x|<2k\eta$, $(1-\eta)^{2k}<(f^j)'(x)<(1-\eta)^{-2k}$ and $\|f^j-f_0\|_{C^k} < (4k+2)\eta$. 

Finally, let us take $j$ with $-2k\le j\le 2k$, $F=h$ and $G=f^j$.
We first do the $C^0$ estimate:
\[
\begin{gathered}
|h(f^j(x)) - h_0(x)| \le
|h(f^j(x)) - h_0(f^j(x))| + |h_0(f^j(x)) - h_0(x)|  < \\
<
\eta + \|h_0\|_{C^1} |f^j(x) - x| \le (1 + |j| \|h_0\|_{C^1}) \eta 
\le (1 + 2kB_{k+1}) \eta.
\end{gathered}
\]
Next the $C^1$ estimate:
we have
\[
\begin{gathered}
|(h\circ f^j)'(x)-h_0'(x)|=|h'(f^j(x))\cdot (f^j)'(x)-h_0'(x)| \le \\
\le |h'(f^j(x))\cdot ((f^j)'(x)-1)| + |h'(f^j(x))-h_0'(f^j(x))|
+|h_0'(f^j(x))-h_0'(x)|,
\end{gathered}
\]
which is at most
\[
\begin{gathered}
(4k+2)\eta\|h\|_{C^1}+\|h-h_0\|_{C^1}+2k\eta\|h_0\|_{C^2}< \\
<(4k+2)\eta(\|h_0\|_{C^2}+\eta)+\eta+2k\eta\|h_0\|_{C^2}= \\
= (6k+2)\eta\|h_0\|_{C^2}+\eta+(4k+2)\eta^2
<( (6k+2) B_{k+1} +2 ) \eta.
\end{gathered} \]
For $r\ge 2$, we have
$(F\circ G)^{(r)}(x)=F^{(r)}(G(x))(G'(x))^r+F'(G(x))G^{(r)}(x)+S$, where $S$ is the sum of the intermediate terms in Fa\`a di Bruno's formula. We thus have
\[ |(F\circ G)^{(r)}(x)-h_0^{(r)}(x)|\le |F^{(r)}(G(x))(G'(x))^r-h_0^{(r)}(x)|+|F'(G(x))G^{(r)}(x)|+|S|. \]

We have
\[
\begin{gathered}
|F^{(r)}(G(x))(G'(x))^r-h_0^{(r)}(x)| \\
\le |F^{(r)}(G(x))((G'(x))^r-1)|+|F^{(r)}(G(x))-h_0^{(r)}(G(x))|+|h_0^{(r)}(G(x))-h_0^{(r)}(x)| \\
\le (k+1)(4k+2)\eta\|h\|_{C^r}+\eta+2k\eta\|h_0\|_{C^{r+1}} \\
< (k+1)(4k+2)\eta(\|h_0\|_{C^r}+\eta)+2k\eta\|h_0\|_{C^{r+1}}+\eta \\
< (4k^2+8k+2)\eta\|h_0\|_{C^{r+1}}+\eta+(4k^2+6k+2)\eta^2  \\
< (4k^2+8k+2)\eta\|h_0\|_{C^{r+1}}+2\eta.
\end{gathered} \]

We also have 
\[
\begin{gathered}
|F'(G(x))G^{(r)}(x)|<(4k+2)\eta\|h\|_{C^1}<(4k+2)\eta(\|h_0\|_{C^1}+\eta)= \\
= (4k+2)\eta\|h_0\|_{C^1}+(4k+2)\eta^2<(4k+2)\eta\|h_0\|_{C^1}+\eta, 
\end{gathered} \]
and
\[
\begin{gathered}
|S|\le 2\|h\|_{C^r}(2+(4k+2)\eta)^r(4k+2)\eta= \\
= 2^{r+2}(2k+1)(1+(2k+1)\eta)^r\eta\|h\|_{C^r} < \\
<2^{k+2}(2k+1)(1+(k+1)(2k+1)\eta)\eta\|h\|_{C^r}< \\
< 2^{k+3}(k+1)\eta\|h\|_{C^r}<2^{k+3}(k+1)\eta(\|h_0\|_{C^r}+\eta)<2^{k+3}(k+1)\eta\|h_0\|_{C^r}+\eta.
\end{gathered} \]
Therefore, we have
\[
\begin{gathered}
|(F\circ G)^{(r)}(x)-h_0^{(r)}(x)|\le  \\
\le (2^{k+3}(k+1)+4k^2+12k+4)\eta\|h_0\|_{C^{r+1}}+4\eta< \\ 
< ((2^{k+3}(k+3)-2k) B_{k+1} +4)\eta,
\end{gathered} \]
which concludes the proof of the lemma.
\end{proof}

\bigbreak

We are ready to prove Proposition~\ref{prop:fh}.
Recall that this proposition shows that the divided difference
$[z_0,z_1,\ldots,z_k;h(\tilde z_0),h(\tilde z_1),\ldots,h(\tilde z_k)]$
is well approximated by $h_0^{(k)}(0)/k!$.



\bigbreak

\begin{proof}[Proof of Proposition \ref{prop:fh}]
Consider $k$, $f$ and $h$ fixed and write $\delta = f(0) > 0$;
let $\tilde h_{j}:=h\circ f^j$ for $0\le j\le k$.
The value of $C_k$ will only be defined at the end of the proof.
We can already assume that 
\[ C_k \ge \max\left\{ \frac{4k}{\epsilon}, (k+1)2^{k+2} \right\} \]
so that 
\[ \eta < \min\left\{ \frac{\epsilon}{8k}, \frac{1}{(k+1)2^{k+3}} \right\}, \]
in particular satisfying the conditions 
of Lemma~\ref{lemma:tildehj}.

Define the function $g$ by:
\begin{align*}
g(\tilde z_0) &= 
[z_0,z_1,\ldots,z_k;h(\tilde z_0),h(\tilde z_1),\ldots,h(\tilde z_k)] \\
&=
[z_0,z_1,\ldots,z_k;
\tilde h_0(\tilde z_0),\tilde h_1(\tilde z_0),\ldots,\tilde h_k(\tilde z_0)];
\end{align*}
notice that $\tilde z_j$ is taken to be a function of $\tilde z_0$.
For $\tilde z_0 = z_j$ (where $j$ is an integer, $-2k \le j \le 2k$),
we have (from Equation~\eqref{equation:zeta})
\[ g(z_j) = 
[z_0,z_1,\ldots,z_k;
\tilde h_j(z_0),\tilde h_j(z_1),\ldots,\tilde h_j(z_k)] = 
\frac{1}{k!} \tilde h_j^{(k)}(\zeta) \]
for some $\zeta \in [z_0,z_k]$.
Notice that if $\zeta \in [z_{-2k},z_{2k}]$ (and $B = h_0^{(k)}(0)$) then
\[ \begin{aligned}
|\tilde h_j^{(k)}(\zeta)-B| &\le
|\tilde h_j^{(k)}(\zeta)-h_0^{(k)}(\zeta)|+|h_0^{(k)}(\zeta)-h_0^{(k)}(0)| \\
&<((2^{k+3}(k+3)-2k)\|h_0\|_{C^{k+1}}+4)\eta+\|h_0\|_{C^{k+1}}\cdot 2k\eta \\
&=c_1\eta, 
\end{aligned} \]
with $c_1:=2^{k+3}(k+3)\|h_0\|_{C^{k+1}}+4$
(in the second line we use Lemma~\ref{lemma:tildehj}).

We then have
\begin{equation}
\label{equation:gzj}
\left|g(z_j) - \frac{B}{k!}\right|=
\frac{1}{k!}\left|\tilde h_j^{(k)}(\zeta)-B\right|
\le \frac{1}{k!}c_1\eta=c_2\eta
\end{equation}
(for $c_2:=\frac{c_1}{k!} > 0$, which depends only on $\|h_0\|_{C^{r+1}}$ and $k$).
Our aim is to prove a similar estimate
for all $\tilde z_0 \in [z_{-2k},z_{2k}]$.

Let $\tilde H_j$ be the following polynomial of degree $k$
(in the variable $z$ or $\tilde z_0$):
\[ \tilde H_j(z) = \tilde h_j(0) + \tilde h_j'(0) z + \cdots +
\tilde h_j^{(i)}(0) \frac{z^i}{i!} + \cdots +
\tilde h_{j}^{(k-1)}(0) \frac{z^{k-1}}{(k-1)!} + \frac{B}{k!}z^k. \]
For $z \in [z_{-2k},z_{2k}]$,
write the Taylor expansion of $\tilde h_j(z)$ around $z = 0$
with Lagrange remainder:
\[
\begin{aligned}
\tilde h_j(z) &= \tilde h_j(0) + \tilde h_j'(0) z + \cdots +
\tilde h_j^{(i)}(0) \frac{z^i}{i!} + \cdots +
\tilde h_j^{(k)}(\zeta) \frac{z^k}{k!} \\
&= \tilde H_j(z) + (\tilde h_j^{(k)}(\zeta) - B) \frac{z^k}{k!}
\end{aligned}
\]
for some $\zeta \in [z_{-2k},z_{2k}]$.
Since we know that $|\tilde h_j^{(k)}(\zeta) - B| < c_1 \eta$
we have
\begin{equation}
\label{equation:Hh}
|\tilde H_j(z) - \tilde h_j(z)|
< \frac{c_1}{k!}|z|^k \eta
= c_2|z|^k \eta
\le c_2 \eta (2k(1+4k\eta)\delta)^k<c_3\eta\delta^k
\end{equation}
for all $z \in [z_{-2k},z_{2k}]$,
where $c_3=2^{k+1}k^k\cdot c_2$
depends only on $k$ and $\|h_0\|_{C^{k+1}}$, not on $\eta$ or $\delta$.
{The last inequality follows from $(1+4k\eta)^k < 2$
(which is true because $\eta < \frac{1}{(k+1)2^{k+3}}$).}


Define the polynomial $G$ (also in the variable $z$ or $\tilde z_0$) by:
\begin{equation}
\label{equation:Gu}
\begin{aligned}
G(z) &=
[z_0,z_1,\ldots,z_k;\tilde H_0(z),\tilde H_1(z),\ldots,\tilde H_k(z)] \\
&= u_0 \tilde H_0(z) + u_1 \tilde H_1(z) + \cdots + u_k \tilde H_k(z)
\end{aligned}
\end{equation}
where, as in Equation~\eqref{equation:expand},
\[ u_j = \prod_{i \ne j} \frac{1}{z_i - z_j}. \]
Similarly, we have
\begin{equation}
\label{equation:gu}
g(z) = u_0 \tilde h_0(z) + u_1 \tilde h_1(z) + \cdots + u_k \tilde h_k(z).
\end{equation}
Since the sequence $(z_j)_{0\le j\le k}$ is a $k\eta$-quasi-AP,
for $i \ne j$ we have
\[ (1-k\eta)|i-j|\delta<|z_i-z_j|<(1+k\eta)|i-j|\delta \]
and therefore
$$|u_j| =
\prod_{i \ne j} \frac{1}{|z_i - z_j|}
\le \frac{(1-k\eta)^{-k}\delta^{-k}}{j!(k-j)!}
=\frac{(1-k\eta)^{-k}}{k!}{k\choose j}\delta^{-k}
<\frac{2}{k!}{k\choose j}\delta^{-k}.$$

Equations~\eqref{equation:Hh}, \eqref{equation:Gu} and \eqref{equation:gu}
then imply
\[ |G(z) - g(z)|
<\frac{2c_4\delta^k}{k!}\delta^{-k}\eta\sum_{j=0}^k {k\choose j}= c_4 \eta \]
for all $z \in [z_{-2k},z_{2k}]$,
where $c_4=2^{k+1}c_3/k!=4^{k+1}k^k\cdot c_2/k!$.
In particular, \eqref{equation:gzj} implies
$|G(z_j) - {B}/{k!}| < (c_2+c_4) \eta$ for all integer $j$,
$-2k \le j \le 2k$.
Apply Lemma~\ref{lemma:interp} with $x_j = z_{4j-2k}$
(an increasing $4k\eta$-quasi-AP), $P = G-{B}/{k!}$
(a~polynomial of degree at most $k$) and 
$C = (c_2+c_4) \eta$ to conclude that $|G(z) - {B}/{k!}| < c_5 \eta$ 
for all $z \in [z_{-2k},z_{2k}]$.
Since $(2(1+4k\eta)/(1-4k\eta))^k<3\cdot 2^{k-1}$,
we may take $c_5=3\cdot 2^{k-1}(c_2+c_4)$.

We therefore have that
(for all $z \in [z_{-2k},z_{2k}]$):
\[ \begin{gathered}
|g(z) - \frac{B}{k!}|
\le |G(z) - g(z)| + |G(z) - \frac{B}{k!}| 
< (c_4+c_5)\eta = \\
= ((3\cdot 2^{k-1}+1)c_4+3\cdot2^{k-1}c_2)\eta
=((3\cdot 2^{k-1}+1)4^{k+1}k^k\cdot
\frac{c_1}{k!^2}+3\cdot2^{k-1}\frac{c_1}{k!})\eta = \\
=\frac{c_1}{k!}((3\cdot 2^{k-1}+1)4^{k+1}\frac{k^k}{k!}+3\cdot2^{k-1})\eta.
\end{gathered} \]
Since $k!>k^ke^{-k}$,
this is smaller than
\[ \begin{gathered}
\frac{c_1}{k!}((3\cdot 2^{k-1}+1)4^{k+1}e^k+3\cdot2^{k-1})\eta
<8^{k+1}e^k\frac{c_1}{k!}\eta = \\
= \frac{\eta}{k!}
(2^{k+3}(k+3)\|h_0\|_{C^{k+1}}+4)
8^{k+1}e^k 
\le 
\frac{\eta}{k!}
(2^{k+3}(k+3) B_{k+1} +4)
8^{k+1}e^k.
\end{gathered} \]
{Define $C_k$ by
\[ C_k = \max\left\{ \frac{4k}{\epsilon},
(2^{k+3}(k+3) B_{k+1} +4)8^{k+1}e^k
\right\}; \]
the proof is complete.}
\end{proof}


\bigbreak
\section{Proof of Theorem~\ref{theo:ck}}
\label{section:proofck}

In this section we present the proof of Theorem~\ref{theo:ck},
starting with the easy implications.
Following \cite{stable}, we say that the base torus is a
\textit{$C^k$-T-unstable orbit}
(resp. \textit{$C^k$-HL-unstable orbit})
if condition 2 (resp. 3) in the statement of Theorem~\ref{theo:ck} holds.

\begin{remark}
\label{remark:THL}
In the statement of Theorem~\ref{theo:ck},
condition 3 trivially implies condition 2.
Equivalently, if an orbit is $C^k$-T-stable
it clearly is also $C^k$-HL-stable.
The reverse implication is the substance of this paper.
\end{remark}

The following proposition shows that condition 1 implies condition 3.

\begin{prop}
\label{prop:1to3}
[Condition 1 implies Condition 3]
\end{prop}

\begin{proof}
First modify the action $\theta$ (resp. $X_j$, $A$)
to obtain $\hat\theta$ (resp. $\hat X_j$, $\hat A$) such that,
for some $\eta \in (0,\delta/2)$ we have:
\begin{itemize}
\item{$\|\hat X_j - X_j\|_{C^k} < \delta/2$;}
\item{for all $w \in D'$ and for all $z \in (-\eta,\eta)$
we have $\hat A(z)w = A(0)w$;}
\item{if $z \in (-\epsilon,\epsilon) \smallsetminus (-\delta,\delta)$
then $\hat A(z) = A(z)$.}
\end{itemize}
{Notice that this is possible since 
$(A^{(j)}(0))w = 0$ for all $j$ ($1 \le j \le k$)
and all $w \in D'$, as in Condition 1.}

Assume for concreteness that $D$ is the vector space of column vectors.
Let $D^\ast$ be the dual of $D$, the space of row vectors.
Let $w^\ast \in D^\ast$ be a dual to $D'$:
for all $w \in D$ we have
$w \in D'$ if and only if $w^\ast w = 0$.
We may assume that
$w^\ast = (w^\ast_1 \cdots w^\ast_n)$
is a unit row vector.
Thus, for $j_0 \ne j_1$ the column vector
\( w_{j_0,j_1} = 
 -w^{\ast}_{j_1} e_{j_0} + w^{\ast}_{j_0} e_{j_1} \)
belongs to $D'$.
Let $\beta: (-\epsilon,\epsilon) \to \RR$ be a bump function
with $\beta(z) \ge 0$ (for all $z \in (-\epsilon,\epsilon)$),
support contained in $(-\eta/2,\eta/2)$, $\beta(0) > 0$
and $\|\beta\|_{C^k} < \delta/2$.
The desired family $(\tilde X_j)$ is
\[ \tilde X_j = \hat X_j + w^\ast_j \beta(z) Z. \]
We verify commutativity:
\[ [\tilde X_{j_0},\tilde X_{j_1}] = 
\beta(z) \left(
[\hat X_{j_0}, w^\ast_{j_1} Z] - 
[\hat X_{j_1}, w^\ast_{j_0} Z]\right)
= \beta(z) \hat A'(z) w_{j_0,j_1}
= 0.\]
We are left with verifying that the orbits are not compact.
{Choose $j_\bullet$ such that $w^\ast_{j_\bullet} \ne 0$;}
assume without loss of generality that $w^\ast_{j_\bullet} > 0$.
For every point $p$ is the base torus, draw a curve
$\gamma: [-1,1] \to H$ of the form
$\gamma(t) = p + t e_{j_\bullet}$.
Lift $\gamma$ to obtain the curve
$\tilde\gamma(t) = (y_1(t), \ldots, y_n(t), z(t))$,
which is contained in the orbit under consideration.
Since the length of $\gamma$ is small,
the initial point is away from the boundary
and $\tilde\theta$ is almost horizontal
it follows that $\tilde\gamma: [-1,1] \to \TT^n \times (-\epsilon,\epsilon)$
is well defined (see Remark~\ref{remark:lift}).
We have $z'(t) > 0$ for all $t$,
implying that $\tilde\gamma(t+1)$ lies strictly above $\tilde\gamma(t)$
for all $t \in [-1,0]$
(i.e., both points have the same coordinate in $\TT^n$
and different coordinates in $(-\epsilon,\epsilon)$).
In particular, the orbit is not compact,
completing the proof.
\end{proof}

\medskip

\begin{remark}
\label{remark:lift}
{
We clarify the concept of \textit{lifting} a curve,
used in \cite{stable} and in the above proof.
Given an almost horizontal local action $\tilde\theta$
of $D$ on $\TT^n \times (-\epsilon,\epsilon)$,
orbits define an almost horizontal foliation.
At a point $p \in \TT^n \times (-\epsilon,\epsilon)$,
the tangent plane to the foliation is spanned by
$\tilde X_1(p), \ldots, \tilde X_n(p)$.
}

There are two surjective maps
$\Pi_M: \TT^n \times (-\epsilon,\epsilon) \to \TT^n$
(projection onto the first coordinate)
and $\Pi_H: H \to \TT^n$
(the universal cover).
Let $J \subset \RR$ be an interval.
A curve $\tilde\gamma: J \to \TT^n \times (-\epsilon,\epsilon)$
is a \textit{lifting} of $\gamma: J \to H$ if
$\tilde\gamma$ is contained in a leaf of the foliation
and $\Pi_M \circ \tilde\gamma = \Pi_H \circ \gamma$.

Given an initial point, lifting exists (and is unique)
provided the lifted curve
$\tilde\gamma$ does not exit the manifold by going too far up or down.
In other words, short curves lift provided we start away from the boundary.
\end{remark}

\bigbreak


We are left with the hard part:
proving that condition 2 
(in Theorem~\ref{theo:ck})
together with
the negation of condition 1 imply a contradiction.
Initially, the construction closely follows
that of Section~2 in \cite{stable}.
The crucial difference appears in the end,
where instead of the mean value theorem
we use divided differences and Proposition~\ref{prop:fh}.
We find the repetition of definitions and of the statements
of a few preliminary results to be helpful for the reader.

We first revise Proposition~2.2 from \cite{stable},
which is a result about almost horizontal foliations of the manifold $M$.
An almost horizontal foliation can described by the normal vector field 
$(c_1(p), \ldots, c_n(p), 1)$
(with $c_i: M \to \RR$ and additional integrability conditions),
or more simply, by the functions $c_i$.
The horizontal foliation $\cF_0$ is given by
$c_1(p) = \cdots = c_n(p) = 0$ (for all $p \in M$).
A foliation $\cF_1$ is \textit{good}
if the corresponding functions $c_i$
are $C^0$-close to $0$ so as to imply several conditions,
which we omit here but are detailed in
Definitions~2.0 and 2.1~in~\cite{stable}.
{The first condition is that
the foliation $\cF_1$ must be almost horizontal.}
Another condition is the following
{(see Remark~\ref{remark:lift}).}
Given a continuous path $\gamma: [0,1] \to H$ of length less than $4n^2$
and an initial point $p_0 = (\gamma(0),z_0)$ with $|z_0| < \epsilon/2$,
$\gamma$ can be lifted to $\tilde\gamma: [0,1] \to M$ 
contained in a leaf of $\cF_1$ and with $\tilde\gamma(0) = p_0$.
A codimension one subspace $W \subset H$ is \textit{rational}
if it admits an equation with rational coefficients;
otherwise, $W$ is irrational.
Notice that if $W$ is rational then its projection
to $\TT^n = H/L$ is a torus $W/(L \cap W) \subset \TT^n$
of dimension $n-1$;
otherwise, it is a dense subset of $\TT^n$.

Let $\cF_1$ be a good perturbation of $\cF_0$
and $p_0 = (y_1,\ldots,y_n,z)$
with $|z| < \epsilon/2$
be such that the leaf passing through
$p_0$ is noncompact.
The proof of Proposition~2.2 from \cite{stable}
presents the construction of a hyperplane $H' \subset H$,
the \textit{good} hyperplane;
we shall revisit this construction below.
The following proposition lists key properties of
the good hyperplane $H' \subset H$.

\begin{prop}
\label{prop:stable22}
Let $\cF_1$ be a good perturbation of $\cF_0$
and $p_0 = (y_1,\ldots,y_n,z)$
with $|z| < \epsilon/2$
be such that the leaf passing through
$p_0$ is noncompact.
Let $H' \subset H$ be the good hyperplane.
The following properties hold:

\begin{enumerate}
\item{The natural affine projection $H' \to \TT^n = H/L$
taking $0$ to $(y_1,\ldots,y_n)$
can be globally lifted
to a function $H' \to M = \TT^n \times (-\epsilon,\epsilon)$
in such a way that $0$ is taken to $p_0$.}
\item{The good hyperplane $H'$ does not change if we replace $p_0$
by another point on the same leaf.}
\item{If $H'$ is irrational
then the closure of the image is a topological $n$-torus $T_0 = T_0(p_0)$
and its projection onto $\TT^n$ is a homeomorphism;
$T_0$ is contained in a union of noncompact leaves.}
\item{If $H'$ is rational
then the supremum and infimum of the intersection of the image
with each vertical line form two (possibly identical)
smooth $(n-1)$-tori $T_1^+ = T_1^+(p_0)$ and $T_1^- = T_1^-(p_0)$
contained in noncompact leaves and such that
their projections onto $H'/(L \cap H') \subseteq \TT^n$
are diffeomorphisms.}
\end{enumerate}
\end{prop}

\begin{remark}
\label{remark:stable22}
The proof of Proposition 2.2 in \cite{stable}
presents an explicit construction of the hyperplane $H'$,
which we review below.
The proof in \cite{stable}
includes the verification of several useful properties,
similar to those stated in Proposition~\ref{prop:stable22}. 

The statement of Proposition 2.2 in \cite{stable}
also includes the statement that $H'$
is the \emph{only} hyperplane with the stated properties.
This uniqueness claim is nowhere used in \cite{stable}.
The statement of Proposition~\ref{prop:stable22}
does not include the uniqueness claim
and in the present paper we neither use
nor prove such claims.
\end{remark}

We find it useful to rediscuss both the construction of $H'$
and the proofs of Proposition~\ref{prop:stable22}
and of Proposition 2.2 in \cite{stable}.
This includes both a sketch of a slightly different proof and
some additional remarks.
We also introduce some objects
which in \cite{stable} first appear in this proof,
notably the functions $f^{\pm}_{i}$ where $1 \le i \le n$.
These functions will be necessary in the proof of Theorem~\ref{theo:ck},
and are closely related to a function $f$
as in the statement of Proposition~\ref{prop:fh}.



Let $\cF_1$ be a good perturbation of $\cF_0$.
The foliation $\cF_1$ is almost horizontal,
in particular transversal to the vertical vector $Z$.
At each point $p$ and for each value of $i$ ($1 \le i \le n$),
the vector $\tilde Y_i = Y_i - c_{i}(p) Z$
is tangent to $\cF_1$ at $p$.
The vector fields $\tilde Y_i = Y_i - c_{i} Z$
are smooth and commute with each other.
Indeed,
we have $D\Pi(p) \tilde Y_i(p) = Y_i(\Pi(p))$
where $\Pi:M \to \TT^n$ is the vertical projection.

For $s \in \{-1,1\}$ and $i$,
integrate the vector fields $s \tilde Y_i$
at time $1$ to obtain a local diffeomorphism
$F_i^s: U_{i,s} \to M$;
here $U_{i,s} \subseteq M$,
$U_{i,s} \supseteq \TT^n \times (-\epsilon/2,\epsilon/2)$,
is a maximal open subset.
{Thus, $F_i^{-1}$ is a local inverse of $F_i^{+1}$,
meaning that $p \in \TT^n \times (-\epsilon/2,\epsilon/2)$
implies $F_i^{-1}(F_i^{+1}(p)) = F_i^{+1}(F_i^{-1}(p)) = p$.
On the other hand, the domains $U_{i,\pm 1}$
are usually strictly contained in $M$,
so that the maps $F_i^{\pm 1}$ are
local diffeomorphisms but not diffeomorphisms.}

Equivalently, given $s$, $i$ and $p \in M$
(try to) construct a smooth curve
$\gamma: [0,1] \to M$ with $\gamma(0) = p$ and 
$\gamma'(t) = s \tilde Y_i(\gamma(t))$:
take $F_i^s(p) = \gamma(1)$.
If $p = (y_1, \ldots, y_n, z)$
then $F_i^s(p) = (y_1, \ldots, y_n, f^s_{i,p}(z))$
where $f^s_{i,p}: J \to (-\epsilon,\epsilon)$
is a smooth increasing function near the identity;
its (maximal) domain is an open interval $J \subseteq (-\epsilon,\epsilon)$,
$J \supseteq (-\epsilon/2,\epsilon/2)$.
If $\cF_1$ is $C^k$-near $\cF_0$
then both $F_i^s$ and $f_i^s$ are $C^k$-near the respective identities.

The functions $f_{i,p}^s$ satisfy $(f_{i,p}^s)^{-1} = f_{i,p}^{-s}$;
more precisely, the inverse $(f_{i,p}^s)^{-1}$ is well defined as,
say, $(f_{i,p}^s)^{-1}: (-\epsilon/3,\epsilon/3) \to (-\epsilon/2,\epsilon/2)$
and coincides in $(-\epsilon/3,\epsilon/3)$ with the restriction
of $f_{i,p}^{-s}$.
Also, the functions $f_{i_0,p}^{s_0}$ and $f_{i_1,p}^{s_1}$ commute
(again after taking appropriate restrictions to an interval
contained in $(-\epsilon/2,\epsilon/2)$).
When possible, we abuse notation
(disregarding the domains) and write $f_{i,p} = f^{+1}_{i,p}$
and interpret $f^{-1}_{i,p}$ to be its inverse.

We may assume that the point in the statement of
Proposition~\ref{prop:stable22} is $p_\bullet = (0,\ldots,0, z_\bullet)$.
Since the leaf of $\cF_1$ containing $(0,\ldots,0,z_\bullet)$ is noncompact,
we may also assume without loss of generality that
$f_{n,p_\bullet}^{-1}(z_\bullet) \le f_{i,p_\bullet}^{-1}(z_\bullet)
\le z_\bullet
\le f_{i,p_\bullet}(z_\bullet) \le f_{n,p_\bullet}(z_\bullet)$;
and that
$f_{n,p_\bullet}^{-1}(z_\bullet) < z_\bullet < f_{n,p_\bullet}(z_\bullet)$.
When the base point is fixed we write $f_i = f_{i,p_\bullet}$.
We shall later apply Proposition~\ref{prop:fh}
with $f = f_{n}$.

Take the interval $(f_n^{-1}(z_\bullet),f_n(z_\bullet))$
and identify $z$ with $f_n^{\pm 1}(z)$
to obtain a copy of $\Ss^1$.
The functions $f_1, \ldots, f_{n-1}$
induce functions $g_i: \Ss^{1} \to \Ss^{1}$ ($1 \le j < n$)
with rotation numbers $h_i \in [0,1]$;
the convention $h_n = 1$ is useful
(indeed, almost unavoidable given
the above identification between $z$ and $f_n(z)$).
Notice that the function $g_i$ commute.
It turns out that an equation for $H'$ is $h_1 y_1 + \cdots + h_n y_n = 0$.
In particular, $H'$ is rational if and only if
$h_i \in \QQ$ for all $j$.

In the irrational case, the proof of Proposition~\ref{prop:stable22}
uses Denjoy's theorem \cite{Denjoy}.
Indeed, Denjoy gives us a continuous and strictly increasing
function $\phi: J \to \RR$
(where $J$ is an open interval with
$(f_n^{-3n}(z_{\bullet}), f_n^{3n}(z_{\bullet})) \subseteq
J \subseteq (-\epsilon,\epsilon)$)
satisfying $\phi(z_\bullet) = 0$ and 
$\phi(f_i^s(z)) = \phi(z) + s h_i$ (for $1 \le i \le n$, $s \in \{\pm 1\}$).
Given $p = (y_1,\ldots,y_n,z) \in M$, $0 \le y_i < 1$, 
let $\gamma:[0,1] \to H$, $\gamma(t) = (1-t) (y_1,\ldots,y_n)$.
Lift $\gamma$ to obtain $\tilde\gamma$ tangent to $\cF_1$,
with $\tilde\gamma(0) = p$ and $\tilde\gamma(1) = (0,\ldots,0,z')$.
The torus $T_0(p_\bullet)$ is defined by the equation
$\phi(z') = h_1 y_1 + \ldots + h_n y_n$:
this torus is a graph of a continuous function from
$\TT^n$ to $(-\epsilon,\epsilon)$.
More generally, if $p = (y_1^\bullet, \ldots, y_n^\bullet, z)$
and $z$ is near $z_\bullet$ then
$T_0(p)$ has equation
\[ \phi(z') - \phi(z) =
h_1 (y_1 - y_1^\bullet) + \ldots + h_n (y_n-y_n^\bullet). \]
The tori $T_0(p)$ thus define a $C^0$ foliation 
of an open neighborhood of $T_0(p_\bullet)$.
Recall that Denjoy's theorem gives us only $\phi \in C^0$
even if the functions $g_i$ are known to be smooth.
Thus, in the irrational case, the tori $T_0(p)$ are (in principle) not smooth.
They are, however, $C^0$ near the horizontal torus.

We are in the rational case
when the functions $g_i: \Ss^1 \to \Ss^1$
have rational rotation numbers.
It is then not necessarily the case
that the functions $g_i$ are conjugate to rotations.
Let $P \subseteq \Ss^1$ be the set of common periodic points
of all the functions $g_i$:
the set $P$ is compact, nonempty and invariant under $g_i$.
Lift $P$ to obtain $\tilde P_{\bullet} \subseteq (-\epsilon,\epsilon)$.
The set $\{p_\bullet\} \times P_{\bullet} \subset M$ is almost invariant
under the diffeomorphisms $F_i^{\pm}$,
meaning that
$F_i^s[U_{i,s} \cap (\{p_\bullet\} \times P_{\bullet})] 
\subseteq \{p_\bullet\} \times P_{\bullet}$.
The tori $T_0^{\pm}$ correspond to performing the construction above
not on $(p_{\bullet},0)$ but on
$(p_{\bullet},z^{\pm})$ where $z^{-} \le 0 \le z^{+}$,
$z^{\pm} \in P_{\bullet}$ and
$z^{+}$ (resp. $z^{-}$) is minimal (resp. maximal)
with the above conditions.

This completes our review of the proof of Proposition~\ref{prop:stable22},
which is about foliations.
We now turn again our attention towards actions.

\bigbreak


Similarly to Equation~\eqref{eq:bt} in Remark~\ref{remark:b},
we can rewrite Equation~\eqref{eq:tildea} as
\begin{equation}
\label{eq:tildeb}
\tilde Y_i(y_1, \ldots, y_n, z) =
\sum_{1 \le j \le n} \tilde b_{ji}(y_1, \ldots, y_n, z)
\tilde X_j(y_1, \ldots, y_n, z).
\end{equation}
Recall that $\tilde Y_i(p) = Y_i - c_i(p) Z$
is tangent to $\cF_1$ (at $p$) and therefore
is a linear combination of the vectors $\tilde X_j(p)$,
consistently with the equation.
The coefficients $\tilde b_{ji}$ define a smooth function
$\tilde B: M \to \RR^{n \times n} = \cL(H;D)$
which is by hypothesis near $B$
(or more precisely, near the modification of the path $B$
to obtain a smooth function 
$B: (M = \TT^n \times (-\epsilon,\epsilon)) \to \RR^{n \times n} = \cL(H;D)$
where the $\TT^n$ coordinate is ignored,
{(see Equation~\eqref{eq:bt}).}
Let $\tilde\tau_p \in \cL(H';D)$ be the restriction of $\tilde B(p)$ to $H'$.

We use vertical projection to define affine structures
and unit measures in $T_0$ or $T_1^+$.
In the irrational case, the projection is a homeomorphism
from $T_0$ to the affine torus $\TT^n$,
which gives us the desired structures.
In the rational case, projection takes $T_1^+$
to an affine codimension one subtorus of $\TT^n$.
This subtorus is spanned by $H'$ and has codimension $1$.
In either case, $\tau = \tau(p_\bullet) \in \cL(H';D)$
is defined by an average (see Definition~2.3 in \cite{stable}):
\begin{equation}
\label{equation:tau}
\tau = \int_{T_0} \tilde\tau_p dp
\quad \textrm{or} \quad
\tau = \int_{T_1^+} \tilde\tau_p dp.
\end{equation}
Since $\tau$ is the average of $\tilde\tau_p$
($p$ in either $T_0$ or $T_1^+$)
and $\tilde\tau_p$ is near $B(z_{\bullet})$
it follows that $\tau$ is an injective linear transformation
near the restriction of $B(z_\bullet)$.
Let $D' \subset D$ be the image of $\tau: H' \to D$,
a subspace (also of codimension $1$).

In order to make use of $\tau$ we need another formula for it.
For this, we define a continuous map $\xi: H' \to D$
($\xi$ is usually not a linear transformation).
In the rational case,
assume for simplicity that $p_\bullet \in T_1^+$.
Given $v \in H'$,
define $\gamma: [0,1] \to H' \subset H$,
$\gamma(t) = tv$.
Lift $\gamma$ to obtain $\tilde\gamma: [0,1] \to M$ tangent to $\cF_0$,
$\tilde\gamma(0) = p_\bullet$.
Notice that $\tilde\gamma$ assumes values in $T_0$ or $T_1^+$.
Define
\[ \xi(v) = \int_0^1 \left( \tilde\tau_{\tilde\gamma(t)} v \right) dt \]
so that $\tilde\theta(\xi(v))(p_\bullet) = \tilde\gamma(1) = (v,z)$
(for some $z$; $\tilde\theta$ is the perturbed action).
If $v_0, v_1 \in H'$, 
define $\gamma_0$ and $\gamma_1$ as above and
$\gamma_{0,1}: [0,1] \to H'$,
$\gamma_{0,1}(t) = (1-t)v_0 + tv_1$.
The curves $\tilde\gamma_i$ are defined as above;
$\tilde\gamma_{0,1}$ is defined so that
$\tilde\gamma_{0,1}(0) = \tilde\gamma_0(1)$ and
$\tilde\gamma_{0,1}(1) = \tilde\gamma_1(1)$.
We then have
\[ \xi(v_1) - \xi(v_0) =
\int_0^1 \left( \tilde\tau_{\tilde\gamma_{0,1}(t)} (v_1 - v_0) \right) dt: \]
this implies that $\xi: H' \to D$ is Lipschitz.
We are ready to review Lemma~2.4 from \cite{stable}:

\begin{lemma}
\label{lemma:stable24}
Assume that the perturbed action $\tilde\theta$ is good.
The linear transformation $\tau$ is injective, constant on orbits
and satisfies $\tau(v) = \lim_{t \to \infty}{{1 \over t}\xi(tv)}$.
\end{lemma}

The proof is omitted; it is based on Birkhoff's theorem \cite{Mane}
and the uniform continuity of $\xi$.
We are ready to prove the last implication
in Theorem~\ref{theo:ck}.


\begin{prop}
\label{prop:2to1}
[Condition 2 implies Condition 1]
\end{prop}

\begin{proof}
Our proof is by contradiction.
We assume that condition 1 does not hold
and that condition 2 holds:
with the help of previous results, we arrive at a contradiction.

Assume that condition 1 does not hold.
We prefer to formulate this negated condition in terms of $H$, not $D$,
as in Remark~\ref{remark:b}.
Let $\cG$ be the (compact) Grassmann space of
hyperplanes $H^\sharp \subset H$.
By hypothesis, for every $H^\sharp \in \cG$
there exists $v_0 \in H^\sharp$ and $j_1$ ($1 \le j_1 \le k$)
such that $(B^{(j_1)}(0))v_0 \ne 0$.
By continuity, there then exist
$\omega_1 \in D^\ast$ (the dual of $D$)
and a compact ball $W_1 \subset H$ with $v_0 \in \interior(W_1)$ and
such that $\omega_1(B^{(j_1)}(0))v_1 > 0$ for all $v_1 \in W_1$.
By compactness of $\cG$, there exists a finite family
of triples $((j_q,\omega_q,W_q))_q$
where $j_q \in \ZZ$, $1 \le j_q \le k$, $\omega_q \in D^\ast$,
$W_q \subset H$ is a compact ball
with the following properties.
For all $H^\sharp \in \cG$ there exists $q$
with $H^\sharp \cap \interior(W_q) \ne \varnothing$.
We also have $\omega_q(B^{(j_q)}(0))v > 0$ for all $v \in W_q$.

Assuming condition 2,
let $(\tilde\theta_\ell)$ be a sequence of good smooth actions
with the following properties.
The $C^k$ distance between $\tilde\theta_\ell$ and
the initial compact action $\theta$ is smaller than $1/2^\ell$
(recall that distances are computed between vector fields
$\tilde X_{\ell,j}$ and $X_j$).
{We may assume
(after perhaps a change of basepoint of the torus)
that there exists a fixed $p_\bullet \in M$
such that, for all $\ell$,
the orbit of $\tilde\theta_\ell$ containing $p_\bullet$
is not compact.}
For each $\ell$, perform the above constructions
to define local diffeomorphisms $F_{\ell,i}$,
real functions $f_{\ell,i}$ and
hyperplanes $H'_{\ell} \subset H$.
By taking subsequences,
we may assume that either all $H'_\ell$ are rational
or that all $H'_\ell$ are irrational.
In the irrational case, we construct tori $T_{\ell,0}$;
in the rational case we construct $T_{\ell,1}^+$.
In either case, 
we construct homomorphisms $\tau_\ell: H'_{\ell} \to D$.

By compactness of $\cG$, we may assume that the sequence $(H'_\ell)$
converges to $H'_\infty \in \cG$.
Take $q_\bullet$ (fixed from now on)
such that $H'_\infty \cap \interior(W_{q_\bullet}) \ne \varnothing$.
By taking a subsequence, we may assume that
$H'_\ell \cap \interior(W_{q_\bullet}) \ne \varnothing$
for all $\ell$.
Take $k_\bullet = j_{q_\bullet} \le k$,
$\omega_\bullet = \omega_{j_{q_\bullet}} \in D^\ast$
and $W_\bullet = W_{q_\bullet} \subset H$.
We may also construct a convergent sequence $(v_\ell)$,
$v_\ell \in H'_\ell \cap \interior(W_\bullet)$
with limit $v_\infty \in H'_\infty \cap W_\bullet$.
We have
$\omega_\bullet(B^{(k_\bullet)}(0))v > 0$
for all $v \in W_\bullet$.
Let $h_0: [-\epsilon,\epsilon] \to \RR$
be the fixed smooth function
$h_0(z) = \omega_\bullet(B(z))v_\infty$:
we have $h_0^{(k_\bullet)}(0) > 0$.
Take real numbers $0 < B_{k_\bullet} \le B_{k_\bullet + 1}$
(fixed from now on)
such that $h_0^{(k_\bullet)}(0) \ge B_{k_\bullet}$
and $\| h_0 \|_{C^{k_\bullet + 1}} \le  B_{k_\bullet + 1}$.
Apply Proposition~\ref{prop:fh}
with $k = k_\bullet$ to obtain $C_{k_\bullet} > 1$
and fix $\eta = (\min\{1,B_{k_\bullet}\})/(2C_{k_\bullet}) > 0$.

We first argue the irrational case,
i.e., $H'_\ell$ irrational for all $\ell$;
we shall later come back for the rational case.
As above, assume for concreteness that
$0 \le h_{\ell,i} \le h_{\ell,n}$ for all $i$, $1 \le i \le n$.
Take $f = f_\ell = f_{\ell,n,p_\bullet}$:
notice that $f(0) > 0$ and that for sufficiently large $\ell$
we have $\| f - f_0 \|_{C^k} < \eta$.
{From now on, such $\ell = \ell_\bullet$ is fixed.}
Define $z_i = f^i(0)$ (for $-4n \le i \le 4n$)
and $p_i = F_n^i(p_\bullet) = (0,\ldots,0,z_i)$.
Notice that $T_{0}(p_i) = F_n^i[T_0]$
form a family of almost parallel and almost horizontal $C^0$ tori.
Let $\tau_i \in \cL(H';D)$ be defined by Equation~\eqref{equation:tau}
for the torus $T_0(p_i)$
{(notice that $H' = H'_{\ell_\bullet}$).}
By Lemma~\ref{lemma:stable24}, $\tau_i = \tau_0$ for all $i$.
In particular, if we define $\bt \in \cL(H';D)$ by
\begin{equation}
\label{equation:bt}
\bt = [z_0,z_1,  \ldots, z_{k_{\bullet}};
\tau_0,\tau_1,  \ldots, \tau_{k_{\bullet}}] 
\end{equation}
we have $\bt = 0$.
We proceed to compute estimates for $\bt$
based on Equation~\eqref{equation:tau}.

\bigbreak

For $y = (y_1, \ldots, y_n) \in [0,1)^n$,
we define a smooth function 
$f_y: (-\epsilon/2,\epsilon/2) \to (-\epsilon,\epsilon)$
{as follows.}
Consider the path $\gamma: [0,1] \to H$,
$\gamma(t) = (y_1, \ldots, y_n)t$.
Given $z \in (-\epsilon/2,\epsilon/2)$,
lift $\gamma$ to define $\tilde\gamma: [0,1] \to M$,
$\tilde\gamma(0) = (0,\ldots,0,z)$: we have
$\tilde\gamma(1) = (y,f_y(z)) = (y_1,\ldots,y_n,f_y(z))$.
Notice that $f_y$ is $C^k$-near the identity.
Define $\zeta: [0,1)^n \to (-\epsilon,\epsilon)$ by
$(y,f_y(\zeta(y))) \in T_0$:
the function $\zeta$ is continuous 
and assumes values near $0$ but is in principle not smooth.
When $y$ is implicit, we write
$\tilde z_i = f^i(\zeta(y))$ so that
$\tilde z_i = f^i(\tilde z_0)$ and
$(y,f_y(\tilde z_i)) \in T_0(p_i)$.
Let $\beta_y: (-\epsilon/2,\epsilon/2) \to \cL(H';D)$
be defined by $\beta_y(z) = \tilde B(y,f_y(z))$:
$\beta_y$ is $C^k$-near $B$
(see Equations~\eqref{eq:bt} and \eqref{eq:tildeb}
for the definition of $\tilde B$ and of $\tilde b_{ji}$).
For $y \in [0,1)^n$, define $h_y: (-\epsilon/2,\epsilon/2) \to \RR$
by $h_y(z) = \omega_\bullet \beta_y(z) v_{\ell_\bullet}$.
By Equations~\eqref{equation:tau} and \eqref{equation:bt}
(and leaving $z_0, \ldots, z_{k_\bullet}$ implicit)
we have
\begin{equation}
\label{equation:bt2}
\begin{aligned}
\omega_\bullet \bt v_{\ell_\bullet} &= \omega_\bullet
\left( \int_{[0,1)^n}
[ \beta_y((\zeta(y))), \beta_y(f(\zeta(y))),\ldots,
\beta_y(f^{k_\bullet}(\zeta(y))) ] \; dy 
\right) v_{\ell_\bullet} \\
&= \int_{[0,1)^n}
[h_y(\tilde z_0), h_y(\tilde z_1), \ldots ,
h_y(\tilde z_{k_\bullet}) ]\; dy.
\end{aligned}
\end{equation}
For sufficiently large $\ell$ we have
$\| h_y - h_0 \|_{C^k} < \eta$ for all $y \in [0,1)^n$.
By Proposition~\ref{prop:fh}
we then have 
\[ \forall y \in [0,1)^n, \quad 
[h_y(\tilde z_0), h_y(\tilde z_1), \ldots ,
h_y(\tilde z_{k_\bullet}) ] > 0. \]
We thus have $\omega_\bullet \bt v_{\ell_\bullet} > 0$,
contradicting $\bt = 0$ and completing
the proof of the irrational case.

The rational case is similar with a few minor adjustments.
We may assume that $(0,\ldots,0,0) \in T_1^+$:
otherwise perform a small vertical translation.
The averaging integral has as domain a subtorus
which may change as a function of $\ell$.
This complicates notation 
but in no significant way affects the argument.
\end{proof}

We recall the structure of the proof of the main theorem.

\begin{proof}[Proof of Theorem~\ref{theo:ck}]
Condition 3 trivially implies condition 2.
The implication $1 \to 3$ is proved in Proposition~\ref{prop:1to3}.
The implication $2 \to 1$ is proved in Proposition~\ref{prop:2to1}.
\end{proof}



\bigbreak




\begin{thebibliography}{10}

\bibitem{BegazoSaldanha2004}
Tania M. Begazo, Nicolau C. Saldanha,
\newblock{Stability of compact actions of the Heisenberg group},
\newblock{arXiv:math/0412501.}

\bibitem{BegazoSaldanha2005}
Tania M. Begazo, Nicolau C. Saldanha,
\newblock{Nilpotent pseudogroups of functions on an interval},
\newblock{Bulletin of the Brazilian Mathematical Society,
Volume 36, pages 25–38, 2005.}

\bibitem{deBoor}
C. de Boor,
\newblock{Divided Differences},
\newblock{Surveys in Approximation Theory, 1, 46--69, 2005;}
\newblock{arXiv:math/0502036.}

\bibitem{Ccoyllo}
Rosa Elvira Quispe Ccoyllo,
\newblock{Estudo da Estabilidade
de um Exemplo de Ação Compacta de Codimensão 2},
\newblock{PhD Thesis, PUC--Rio, 2006.}

\bibitem{Denjoy}
A.~Denjoy,
\newblock{Sur les courbes d\'efinies par les \'equations diff\'erentielles
\`a la surface du tore},
\newblock{J.~Math.~Pure et Appl.~{\bf 11}, ser. 9, 1932.}

\bibitem{LMS}
Jean~Lelis, Carlos~Gustavo~Moreira, Elaine~Silva,
\newblock{On $C^k$-functions mapping $\mathbb Q$ into itself 
and Mahler's problem on Liouville numbers},
\newblock{arXiv, 2026.}


\bibitem{Mane}
R. Ma\~n\'e,
\newblock{Ergodic theory and differentiable dynamics},
\newblock{Springer-Verlag, Berlin -- New York, 1987.}




\bibitem{stable}
Nicolau C. Saldanha,
\newblock{Stability of Compact Actions of $\RR^n$ of Codimension One,}
\newblock{Commentarii Mathematici Helvetici,}
\newblock{Volume: 69, Issue: 3, page 431--446 (1994)}

\end{thebibliography}

\medskip

\noindent
\footnotesize
Carlos Gustavo T. de A. Moreira \\
IMPA, Estr. Dona Castorina 110,
Rio de Janeiro, RJ 22460-320, Brazil.  \\
\url{gugu@impa.br}

\noindent
\footnotesize
Nicolau C. Saldanha \\
Departamento de Matem\'atica, PUC-Rio, \\
R. Marqu\^es de S. Vicente 255,
Rio de Janeiro, RJ 22451-900, Brazil.  \\
\url{saldanha@puc-rio.br}

\end{document}